\documentclass{amsart}
\usepackage{amssymb,latexsym,graphics,verbatim}
\usepackage{graphicx}

\theoremstyle{plain}
\newtheorem{theorem}{Theorem}
\newtheorem{lemma}{Lemma}

\theoremstyle{definition}
\newtheorem{definition}{Definition}
\newtheorem*{problem}{Problem}

\numberwithin{equation}{section}

\begin{document}
\title{Polynomials with Maximum Lead Coefficient bounded on a Finite Set}
\author{Karl Levy}
\email{Karl.Ethan.Levy@Gmail.com}
\date{\today}

\begin{abstract}
What is the maximum possible value of the lead coefficient of a degree $d$ polynomial $Q(x)$ if $|Q(1)|,|Q(2)|,\ldots,|Q(k)|$ are all less than or equal to one?  More generally we write $L_{d,[x_k]}(x)$ for what we prove to be the unique degree $d$ polynomial with maximum lead coefficient when bounded between $1$ and $-1$ for $x\in [x_k]=\{x_1,\cdots,x_k\}$. We  calculate explicitly the lead coefficient of $L_{d,[x_k]}(x)$ when $d\leq 4$ and the set $[x_k]$ is an arithmetic progression. We give an algorithm to generate $L_{d,[x_k]}(x)$ for all $d$ and $[x_k]$.
\end{abstract}

\maketitle
\tableofcontents

\begin{section}{Introduction}

The maximum density of sets that avoid arithmetic progressions is a long standing problem \cite{Roth}\cite{Szemeredi}.  In  O'Bryant's paper \cite{O'Bryant} a technique for constructing a lower bound on this density involves building sets which try to avoid having $k$-term progressions map into an interval in the image of a degree $d$ polynomial.  This has lead to the question: what is the maximum possible value of the lead coefficient of a degree $d$ polynomial $Q(x)$ when $|Q(1)|,|Q(2)|,$ $\ldots,|Q(k)|$ are all less than or equal to one?

\begin{theorem}\label{Thm:upperBounds}
If $M>0$ and $Q(x)$ is a polynomial with degree $d$ such that
\[|Q(x)|=|a_dx^d+\cdots+a_1x+a_0|\leq M\]
for all $x$ in the $k$-term arithmetic progression $x_1, x_1+\Delta,...,x_1+(k-1)\Delta$ with $k>d$, then the lead coefficient $a_d$ of $Q(x)$ satisfies for 

\begin{align*}
d=1\text{, }&a_1\leq \frac{M}{\Delta}\frac{2}{k-1}\\
d=2\text{, }&a_2\leq \begin{cases} \dfrac{M}{\Delta^2}\dfrac{8}{(k-1)^2}\scriptstyle\text{ for }k\equiv 1\bmod 2\\ \dfrac{M}{\Delta^2}\dfrac{8}{k(k-2)}\scriptstyle \text{ for }k\equiv 0\bmod 2 \end{cases}\\
d=3\text{, }&a_3\leq \begin{cases} \dfrac{M}{\Delta^3}\dfrac{32}{(k-1)^3}&\scriptstyle \text{ for }k\equiv 1\bmod 4\\ \dfrac{M}{\Delta^3}\dfrac{32}{k(k-1)(k-2)}&\scriptstyle \text{ for } k\equiv 2\bmod 4\\ \dfrac{M}{\Delta^3}\dfrac{32}{(k+1)(k-1)(k-3)}&\scriptstyle \text{ for } k\equiv 3\bmod 4\\ \dfrac{M}{\Delta^3}\dfrac{32}{k(k-1)(k-2)}&\scriptstyle \text{ for } k\equiv 0\bmod4\end{cases} \\
d=4\text{, }&a_4\leq\begin{cases}\dfrac{M}{\Delta^4}\cdot\min\limits_{ x\in I}\left\{\dfrac{8}{x^2(k-1)^2-4x^4}\right\} &\scriptstyle 
\text{ for } k\equiv 1\bmod 2 \text{ and } I=\{\lceil {\frac{k-1}{2\sqrt{2}}}\rceil,\lfloor {\frac{k-1}{2\sqrt{2}}}\rfloor\} \\ \dfrac{M}{\Delta^4}\cdot\min\limits_{ x\in H}\left\{\dfrac{-32}{16x^4-4((k-1)^2+1)x^2+(k-1)^2}\right\} &\scriptstyle \text{ for } k\equiv 0\bmod 2 \text{ and } H=\{\frac{1}{2}\lceil{\frac{k-1}{\sqrt{2}}}\rceil,\frac{1}{2}\lfloor{\frac{k-1}{\sqrt{2}}}\rfloor\}\end{cases} \\
\end{align*}
and for $d\geq5$ we can use the algorithm described in section \ref{algorithmSection} to find the upper bound for $a_d$.
\end{theorem}
Moreover, the inequalities in Theorem \ref{Thm:upperBounds} are sharp and there is a unique and computable polynomial for which equality is achieved.

\begin{theorem}
Given a $k$-term arithmetic progression 
\[[x_k]=\{x_1, x_1+\Delta,...,x_1+(k-1)\Delta\},\] 
for every $d<k$ there is a unique and computable polynomial 
\[|L_{d,[x_k]}(x)|=|a_dx^d+\cdots+a_1x+a_0|\]
 with maximum lead coefficient $a_d$ such that for all $x\in[x_k]$ we have that
\[|L_{d,[x_k]}(x)|\leq 1.\]  Moreover, we have for

\begin{align*}
d=1\text{, }&a_1= \frac{M}{\Delta}\frac{2}{k-1}\\
d=2\text{, }&a_2= \begin{cases} \dfrac{M}{\Delta^2}\dfrac{8}{(k-1)^2}\scriptstyle\text{ for }k\equiv 1\bmod 2\\ \dfrac{M}{\Delta^2}\dfrac{8}{k(k-2)}\scriptstyle \text{ for }k\equiv 0\bmod 2 \end{cases}\\
d=3\text{, }&a_3= \begin{cases} \dfrac{M}{\Delta^3}\dfrac{32}{(k-1)^3}&\scriptstyle \text{ for }k\equiv 1\bmod 4\\ \dfrac{M}{\Delta^3}\dfrac{32}{k(k-1)(k-2)}&\scriptstyle \text{ for } k\equiv 2\bmod 4\\ \dfrac{M}{\Delta^3}\dfrac{32}{(k+1)(k-1)(k-3)}&\scriptstyle \text{ for } k\equiv 3\bmod 4\\ \dfrac{M}{\Delta^3}\dfrac{32}{k(k-1)(k-2)}&\scriptstyle \text{ for } k\equiv 0\bmod4\end{cases} \\
d=4\text{, }&a_4= \begin{cases}\dfrac{M}{\Delta^4}\cdot\min\limits_{ x\in I}\left\{\dfrac{8}{x^2(k-1)^2-4x^4}\right\} &\scriptstyle 
\text{ for } k\equiv 1\bmod 2 \text{ and } I=\{\lceil {\frac{k-1}{2\sqrt{2}}}\rceil,\lfloor {\frac{k-1}{2\sqrt{2}}}\rfloor\} \\ \dfrac{M}{\Delta^4}\cdot\min\limits_{ x\in H}\left\{\dfrac{-32}{16x^4-4((k-1)^2+1)x^2+(k-1)^2}\right\} &\scriptstyle \text{ for } k\equiv 0\bmod 2 \text{ and } H=\{\frac{1}{2}\lceil{\frac{k-1}{\sqrt{2}}}\rceil,\frac{1}{2}\lfloor{\frac{k-1}{\sqrt{2}}}\rfloor\}\end{cases} \\
\end{align*}
and for $d\geq5$ we can use the algorithm described in section \ref{algorithmSection} to find $a_d$.
\end{theorem}
  
In general we will call polynomials $L$-polynomials if they have maximum possible lead coefficient while their absolute value is bounded on some finite set. More specifically
\begin{definition}\label{Def:LPolynomials}
Given $k>d\geq 1$ and $[x_k]=\{x_1,x_2,\cdots,x_k\}\subset\mathbb{R}$ then
\[ L_{d,[x_k]}(x) \]
denotes the unique degree $d$ polynomial with maximum lead coefficient such that when $x\in[x_k]$ 
\[|L_{d,[x_k]}(x)|\leq 1.\]
\end{definition}
The uniqueness of this polynomial is proved by Theorem \ref{Thm:Unique} in Section \ref{Sec:MoreTheorems}.  In Sections \ref{Sec:LeadCoefficients} and \ref{Sec:TermsCheb} we will be concerned with $L_{d,[k]}$, here $[k]=\{1,2,\cdots,k\}$.
 
\subsection{A brief summary}
In Section \ref{Sec:ChebPolynomials} we discuss the connection between $L$-polynomials and Chebyshev $T$-polynomials ($L$-polynomials being a discrete analog of Chebyshev $T$-polynomials).  In Section \ref{Sec:CombinatorialPerspective} the problem is recast from the perspective of combinatorial geometry.  From this perspective we rule out the existence of $L$-polynomials when $k\leq d$ and prove theorems that will then be used: to prove the uniqueness of all $L_{d,[x_k]}$, to compute the lead coefficients of some $L_{d,[k]}$, and to describe an algorithm that generates all $L_{d,[x_k]}$.  In Section \ref{Sec:MoreTheorems} more such theorems are proved (but without the problem being recast from the perspective of combinatorial geometry) and the uniqueness of $L_{d,[x_k]}$ is then proved.  In Section \ref{Sec:LeadCoefficients} the lead coefficients of $L_{d,[k]}$ (when $d\leq 4$ and all $k>d$) are calculated and an algorithm that generates all $L_{d,[x_k]}$ (when $k>d\geq1$) is  described.  In Section \ref{Sec:TermsCheb} we write some $L_{d,[k]}$ in terms of corresponding degree $d$ Chebyshev $T$-polynomials.  

\end{section}
\pagebreak

\begin{figure}[h]
\begin{center}
\begin{picture}(250,175)
\includegraphics{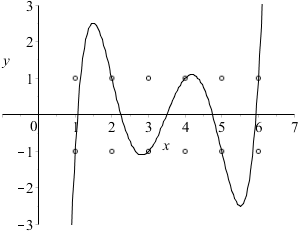}
\end{picture}
\end{center}
\caption{$L_{5,[6]}$ \label{fig:L picture d5}}
\end{figure}

\begin{figure}[h]
\begin{center}
\begin{picture}(250,175)
\includegraphics{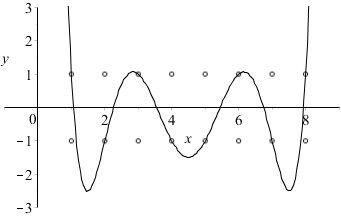}
\end{picture}
\end{center}
\caption{$L_{6,[7]}$ \label{fig:L picture d6}}
\end{figure}

\begin{section}{Chebyshev Polynomials}\label{Sec:ChebPolynomials}
Chebyshev Polynomials $T_d(x)$ are defined by
\[T_0(x)=1,\]
\[T_1(x)=x,\]
\[T_{d+1}(x)=2xT_{d}(x)-T_{d-1}(x).\]
Alternatively, each $T_d(x)$ is the unique polynomial of degree $d$ satisfying the relationship
\[T_d(x)=\cos(d\arccos(x)).\]
Substituting $\cos(x)$ for $x$ this becomes
\[T_d(\cos(x))=\cos(dx).\]
The first few $T_d(x)$ polynomials are
\begin{align*}
T_0(x)=&1\\
T_1(x)=&x\\
T_2(x)=&2x^2-1\\
T_3(x)=&4x^3-3x\\
T_4(x)=&8x^4-8x\\
T_5(x)=&16x^5-20x^3+5x.
\end{align*}

An elementary proof \cite{Rivlin} shows
\[f(x)=\frac{1}{2^{d-1}}T_d(x)\]
to be the unique monic polynomial of degree $d$ with minimum infinity norm on the interval $[-1,1]$.  To wit
$\|f(x)\|_\infty = \frac{1}{2^{d-1}}$.  And thus $T_d$ is the polynomial of degree $d$ that, while bounded between $-1$ and $1$ on the interval $[-1,1]$, has the maximum possible lead coefficient (equal to  $2^{d-1}$, as indicated above by the previous two equations).  And thus our $L$-polynomials, which are bounded between $-1$ and $1$ on a set of $k$ values, can be viewed as a discrete analog of the continuously bounded $T$-polynomials.

We can stretch the Chebyshev polynomial $T_d$ from the interval $[-1,1]$ to the interval $[x_1,x_k]$ by composing it with the following bijection from $[x_1,x_k]$ to $[-1,1]$
\[s(x)=\frac{2x-(x_k+x_1)}{x_k-x_1},\]
and thus from the equations above we get the lead coefficient of $T_d(s(x))$ is 
\[a_d=2^{d-1}\left(\frac{2}{x_k-x_1}\right)^d=\frac{2^{2d-1}}{(x_k-x_1)^d}.\]
This is a lower bound for the maximum lead-coefficient we are looking for since $|T(s(x))|\leq 1$ for $x\in[x_1,x_k]$ are stronger constraints than $|L_{d,[k]}(x)|\leq 1$ for $x\in\{x_1,x_2,\cdots,x_k\}=[x_k]$.

\end{section}

\begin{section}{A Combinatorial Geometry Perspective on $L$-Polynomials}\label{Sec:CombinatorialPerspective}
The problem starts looking purely combinatorial when we remember that polynomials can be factored.  

Consider a polynomial written in terms of its factors
\[Q(x)=a_d\prod_{i\in [d]}(x-r_i)\]
where $d=\deg{(Q(x))}$ and $r_i\in\mathbb{C}$.

Now the question is:  how big can $a_d$ be while, for all $x\in[x_k]$, the following inequality holds
\[|Q(x)|=\left|a_d\prod_{i\in [d]}(x-r_i)\right|=a_d\prod_{i\in [d]} |(x-r_i)|\leq 1?\]
Thus, finding the maximum value of $a_d$ is equivalent to finding the minimum value of the maximum product of distances
\[\max_{x\in [x_k]}\left\{\prod_{i\in [d]} |(x-r_i)|\right\}\]
for all \textit{multisets} $\{r_1,\cdots,r_d\}\subset\mathbb{C}$.  Put concisely:
\begin{problem}
For a given $[x_k]=\{x_1,x_2,\cdots,x_k\}\subset \mathbb{R}$ find the minimum
\[\min_{|R|=d}\left\{\max_{x\in [x_k]}\left\{\prod_{r_i\in R} |x-r_i|\right\}\right\}\]
taken over all \em{multisets} $R=\{r_1,\cdots,r_d\}\subset \mathbb{C}$.
\end{problem}
This minimum is equal to $\frac{1}{a_d}$ where $a_d$ is the lead coefficient of our $L_{d,[x_k]}$.

We do not necessarily have to find the minimizing set $R$ of roots  in order to find this minimum value (i.e. the reciprocal of our maximum lead coefficient).  Though, perhaps it will be useful to consider how a minimizing multiset of roots $R$ must look. We do so and continue by ruling out the existence of a maximum lead coefficient for certain cases of our problem.

\subsection{When $k\leq d$}
Construct a multiset of ``roots'' $R$ such that $[x_k]\subset R$ and thus all products of distances for $x \in [x_k]$ will be $0$ . That is, for example, define
\[\frac{|Q(x)|}{a_d}=|x-x_1|^{1+(d-k)}|x-x_2||x-x_3|\cdots|x-x_k|,\]
then for all $x\in [x_k]$
\[|Q(x)|=0\leq 1.\]
Since $|Q(x)|$ is equal to zero no matter what value is picked for $a_d$, there is no maximum $a_d$.  From here on we are only concerned with the cases where $k>d$.

\subsection{Some useful theorems}
Now we prove a few lemmas using this combinatorial geometry view.

\begin{lemma}\label{Thm:RealRoots}
If for $k>d\geq2$ 
\[Q(x)=a_d\prod_{i\in [d]}(x-r_i)\]
is a degree $d$ polynomial such that for a given
\[[x_k]=\{x_1< x_2<\cdots< x_k\}\]
we have that 
\begin{equation}\label{Exp:MinMax}
\max_{x\in[x_k]}\left\{\prod_{i\in[d]} |x-r_i|\right\}
\end{equation}
is minimal then the multiset $\{r_1,\cdots,r_d\}$ of all $d$ roots of $Q(x)$ is contained in $\mathbb{R}$.\end{lemma}
\begin{proof} Assume that $r_i=a_i+(\sqrt{-1})b_i$ with $b_i\neq0$ for some $i$.  This means that for all $x\in[x_k]$
\[|x-r_i|=\sqrt{(x-a_i)^2 + (b_i)^2}>\sqrt{(x-a_i)^2}=|x-a_i|.\]
But then
\begin{equation*}
\max_{x\in[x_k]}\left\{\prod_{i\in[d]} |x-r_i|\right\}>\max_{x\in[x_k]}\left\{\prod_{i\in[d]} |x-a_i|\right\},\end{equation*}
contradicting the assumed minimality of the expression (\ref{Exp:MinMax}) on the left-side of this inequality.  It follows then that 
\[\{r_1,\cdots,r_d\}=\{a_1,\cdots,a_d\}\subset\mathbb{R}.\]
\end{proof}
Thus we have proved that the roots of our $L$-polynomials are all real.

\begin{theorem}\label{Thm:DistinctRoots} If for $k>d\geq2$
\[Q(x)=a_d\prod_{i\in [d]}(x-r_i)\]
is a degree $d$ polynomial such that for a given
\[[x_k]=\{x_1< x_2<\cdots< x_k\}\]
we have that 
\begin{equation}\label{Exp:MinMax2}
\max_{x\in[x_k]}\left\{\prod_{i\in[d]} |x-r_i|\right\}
\end{equation}
is minimal then the multiset of the $d$ real (Lemma \ref{Thm:RealRoots}) roots of $Q(x)$ contains no duplicates for $k>d\geq2$.
\end{theorem}

\begin{proof}Assume $Q(x)$ has a root of multiplicity greater than one.  If we make no claim about the the ordering of the $r_i$ we can, without loss of generality, set 
\[r=r_1=r_2\]
and rewrite the product of distances (i.e. the factors in our polynomial) as
\[|x-r|^2|x-r_3|\ldots|x-r_d|.\]
We will show that replacing the two identical $r$'s with two distinct roots, $r+\varepsilon$ and $r-\varepsilon$, yields a smaller maximum product (\ref{Exp:MinMax2}), a contradiction.  We choose $\varepsilon$ to be the minimum  of a subset formed from $\varepsilon_1$, $\varepsilon_2$ and $\varepsilon_3$, each of which is defined for a separate case of $x$:
\begin{description}
\item[Case I] Set \[\varepsilon_1=\min\left\{\frac{x-r}{2} : x\in[x_k]\text{ and } x>r \right\}\]
and thus when $x>r$ and $x\in{[}x_k{]}$
\begin{align*}
		|x-(r+\varepsilon_1)||x-(r-\varepsilon_1)|&=(x-r)^2-\varepsilon_1^2\\
		&<|x-r|^2.
\end{align*}
Now multiplying both sides of the inequality by the other roots' distances to $x$ (i.e. the absolute values of $Q(x)$'s factors) yields
\begin{align*}
|x-(r+\varepsilon_1)||x-(r-\varepsilon_1)||x-r_3|\ldots|x-r_d|<\\
|x-r|^2|x-r_3|\ldots|x-r_d|
\end{align*}
when $x>r$ and $x\in{[}x_k{]}$.
\item[Case II]  Set \[\varepsilon_2=\min\left\{\frac{r-x}{2}: x\in[x_k]\text{ and } x<r\right\}\]
and thus when $x<r$ and $x\in{[}x_k{]}$ 
	\begin{align*}
		|x-(r+\varepsilon_2)||x-(r-\varepsilon_2)|&=(r-x)^2-\varepsilon_2^2\\
		&<|x-r|^2.
	\end{align*}
Now multiplying both sides of the inequality by the other roots' distances to $x$ (i.e. the absolute values of $Q(x)$'s factors) yields
\begin{align*}
|x-(r+\varepsilon_2)||x-(r-\varepsilon_2)||x-r_3|\ldots|x-r_d|<\\
|x-r|^2|x-r_3|\ldots|x-r_d|
\end{align*}
when $x<r$ and $x\in{[}x_k{]}$.
\item[Case III] If $r\in[x_k]$ and $r$ is a root of multiplicity two set 
\[\varepsilon_3=\sqrt{\frac{1}{2}\frac{\max_{x\in [x_k]}\left\{|x-r|^2|x-r_3|\cdots|x-r_d|\right\}}{|r-r_3|\cdots|r-r_d|}}\]
and thus when $x=r$ and $x\in{[}x_k{]}$
\begin{align*}	
|x-(r+\varepsilon_3)||x-(r-\varepsilon_3)|=&\varepsilon_3^2\\
=&\frac{1}{2}\frac{\max_{x\in [x_k]}\left\{|x-r|^2|x-r_3|\cdots|x-r_d|\right\}}{|r-r_3|\cdots|r-r_d|}\\
<&\frac{\max_{x\in [x_k]}\left\{|x-r|^2|x-r_3|\cdots|x-r_d|\right\}}{|r-r_3|\cdots|r-r_d|}.
\end{align*}
Now multiplying both sides of the inequality by the other roots' distances to $x=r$ (i.e. the absolute values of $Q(x)$'s factors) yields
\begin{multline*}
|x-(r+\varepsilon_3)||x-(r-\varepsilon_3)||x-r_3|\ldots|x-r_d|<\\
\max_{x\in [x_k]}\left\{|x-r|^2|x-r_3|\cdots|x-r_d|\right\}
\end{multline*}
when $x=r$ and $x\in{[}x_k{]}$	
\end{description}
Now we choose $\varepsilon$.  If $r\in[x_k]$ and $r$ is a root of multiplicity two set $\varepsilon=\min\left\{\varepsilon_1,\varepsilon_2,\varepsilon_3\right\}$, otherwise set $\varepsilon=\min\{\varepsilon_1,\varepsilon_2\}$. Using $\varepsilon$ in the final inequality of each of the above three cases yields
\begin{multline*}
\max_{x\in [x_k]}\left\{|x-(r+\varepsilon)||x-(r-\varepsilon)||x-r_3|\cdots|x-r_d|\right\}<\\
\max_{x\in [x_k]}\left\{|x-r|^2|x-r_3|\cdots|x-r_d|\right\},
\end{multline*}
contradicting the minimality of (\ref{Exp:MinMax2}).
\end{proof}
Thus we have proved that the roots of our $L$-polynomials are distinct.

\begin{lemma}\label{Thm:RootsInterval} If for $k>d\geq2$ 
\[Q(x)=a_d\prod_{i\in [d]}(x-r_i)\]
is a degree $d$ polynomial such that for a given
\[[x_k]=\{x_1< x_2<\cdots< x_k\}\]
we have that 
\begin{equation}\label{Exp:MinMax3}
\max_{x\in[x_k]}\left\{\prod_{i\in[d]} |x-r_i|\right\}
\end{equation}
is minimal then the set $\{r_1<r_2<\cdots<r_d\}$ of $d$ distinct (Theorem \ref{Thm:DistinctRoots}) roots of $Q(x)$ is contained in the open interval $(x_1,x_k)$.
\end{lemma}

\begin{proof}
Suppose that one of the roots is outside of the open interval $(x_1,x_k)$, that is for some $\delta\geq 0$ we have $r_d=x_k+\delta$.  We will show that replacing the root $r_d$ with $x_k-\varepsilon$ for some $\varepsilon>0$ yields a smaller maximum product (\ref{Exp:MinMax3}), a contradiction.  And we will choose $\varepsilon$ to be the minimum  of $\varepsilon_1$ and $\varepsilon_2$ which are now defined for two cases of $x$:
\begin{description}
\item[Case I] Set 
\[\varepsilon_1=\frac{x_k-x_{k-1}}{2}\]
and thus for any $\delta\geq0$ and all $x\in\{x_1,\cdots,x_{k-1}\}=[x_{k-1}]$ 
\begin{align*}
|x-r_d| = |x-(x_k+\delta)| = (x_k+\delta)-x > \left(\frac{x_k+x_{k-1}}{2}\right)-x=\\
\left|x-\left(x_k-\frac{x_k-x_{k-1}}{2}\right)\right|=|x-(x_k-\varepsilon_1)|.
\end{align*}
Now multiplying both sides of the inequality by the other roots' distances to $x$ (i.e. the absolute values of $Q(x)$'s factors) yields
\[|x-r_1|\cdots|x-r_{d-1}||x-r_d|>|x-r_1|\cdots|x-r_{d-1}||x-(x_k-\varepsilon_1)|\]
for $x\in[x_{k-1}]$.
\item[Case II] Set 
\[\varepsilon_2=\frac{1}{2}\frac{\max_{x\in [x_k]}\{|x-r_1|\cdots|x-r_d|\}}{|x_k-r_1|\cdots|x_k-r_{d-1}|}\]
and thus for $x=x_k$
\begin{align*}	
|x-(x_k-\varepsilon_2)|=&\varepsilon_2\\
=&\frac{1}{2}\frac{\max_{x\in [x_k]}\{|x-r_1|\cdots|x-r_d|\}}{|x_k-r_1|\cdots|x_k-r_{d-1}|}\\
<&\frac{\max_{x\in [x_k]}\{|x-r_1|\cdots|x-r_d|\}}{|x_k-r_1|\cdots|x_k-r_{d-1}|.}
\end{align*}
Now multiplying both sides of the inequality by the other roots' distances to $x=x_k$ (i.e. the absolute values of $Q(x)$'s factors) yields
\begin{align*}
|x-r_1|\ldots|x-r_{d-1}||x-(x_k-\varepsilon_2)|<
\max_{x\in [x_k]}\{|x-r_1|\cdots|x-r_d|\}
\end{align*}
for $x=x_k$.	
\end{description}
Now set $\varepsilon=\min\{\varepsilon_1,\varepsilon_2\}$ and thus using $\varepsilon$ in the final inequality of each of the above two cases yields
\begin{align*}
\max_{x\in [x_k]}\{|x-r_1|\ldots|x-r_{d-1}||x-(x_k-\varepsilon)|\}<
\max_{x\in [x_k]}\{|x-r_1|\cdots|x-r_d|\}.
\end{align*}
contradicting the minimality of (\ref{Exp:MinMax3}).
The same arguments will work for attempting to place a root to the left of the interval $(x_1,x_k)$ (i.e. if for some $\delta\geq 0$ we have $r_1=x_1-\delta$).
\end{proof}

Thus we have proved that the roots of our $L$-polynomials are in the interval $(x_1,x_k)$.  We can now say that the roots of our $L$-polynomials are distinct real numbers of multiplicity one in the interval $(x_1,x_k)$ by putting Lemma \ref{Thm:RealRoots}, Theorem \ref{Thm:DistinctRoots} and Lemma \ref{Thm:RootsInterval} together.

\end{section}

\begin{section}{Some More Useful Theorems (Not from the Combinatorial Geometry Perspective  )}\label{Sec:MoreTheorems}

First we show that a degree $d$ $L$-polynomial bounded between $-1$ and $1$ for $x\in[x_k]$ must pass through the points $(x_1,(-1)^d)$ and $(x_k,1)$, on both sides of the set of boundary points $\{(x,\pm1):x\in[x_k]\}$

\begin{lemma}\label{Thm:TerminalPoints} 
If for $k>d\geq2$
\[Q(x)=a_d\prod_{i\in [d]}(x-r_i)\]
is a degree $d$ polynomial with maximum lead coefficient $a_d$ such that
\[|Q(x)|\leq 1\]
when
\[x\in[x_k]=\{x_1< x_2<\cdots< x_k\}\]
then
 \[Q(x_1)=(-1)^d\] and \[Q(x_k)=1.\] \end{lemma}
 
\begin{proof} Let us suppose that $Q(x_k)<1$. We know that $Q(x)$ has $d$ distinct roots and that they are contained in the interval $(x_1,x_k)$, from Theorem \ref{Thm:DistinctRoots} and Lemma \ref{Thm:RootsInterval} respectively.  In other words, for some
\[x_1<r_1<r_2<\cdots<r_{d-1}<r_d<x_k,\]
we have
\[Q(x)=a_d(x-r_1)\cdots(x-r_d).\]
Now with some soon to be determined $\varepsilon$ we define 
\[\hat{Q}(x)=Q(x)+\varepsilon (x-r_1)\cdots(x-r_{d-1}).\]
So for $x\neq r_d$ we can write
\[\hat{Q}=\left(1+\frac{\varepsilon}{a_d(x-r_d)}\right)Q(x).\]
We will use this $\hat{Q}(x)$ to contradict the maximality of $Q(x)$'s lead coefficient.  We will choose $\varepsilon$ to be the minimum  of a subset of $\varepsilon_1$, $\varepsilon_2$ and $\varepsilon_3$, which are now defined for three cases of $x$:

\begin{description}
\item[Case I] Set
\[\varepsilon_1=\frac{a_d}{2}(\min\left\{(r_d-x):x<r_d\text{ and }x\in[x_k]\right\})\]
and thus for $x<r_d$ and $x\in[x_k]$
\begin{multline*}
\left|\left(1+\frac{\varepsilon_1}{a_d(x-r_d)}\right)Q(x)\right|=\\
\left|\left(1-\frac{1}{2}\frac{(\min\left\{(r_d-x):x<r_d\text{ and }x\in[x_k]\right\})}{(r_d-x)}\right)Q(x)\right|<\left|Q(x)\right|\leq 1
\end{multline*}
  
\item[Case II]Since $Q(x_k)>0$ by Lemma \ref{Thm:RootsInterval}, we can set
\[\varepsilon_2=\frac{a_d}{2}\left(\frac{1}{Q(x_k)}-1\right)(\min\left\{\left(x-r_d)\right):x>r_d\text{ and }x\in[x_k]\right\}).\]
Recall that we assumed $Q(x_k)<1$ and thus for $x>r_d$ and $x\in[x_k]$
\begin{multline*}
\left|\left(1+\frac{\varepsilon_2}{a_d(x-r_d)}\right)Q(x)\right|=\\
\left(1+\frac{1}{2}\frac{(\min\left\{(x-r_d):x>r_d\text{ and }x\in[x_k]\right\})}{(x-r_d)}\left(\frac{1}{Q(x_k)}-1\right)\right)\left|Q(x)\right|\leq\\
\left(1+\frac{1}{2}\left(\frac{1}{Q(x_k)}-1\right)\right)\left|Q(x)\right|.
\end{multline*}
Now since $Q(x)$ is a polynomial with positive lead coefficient, it increases to the right of its largest root $r_d$. So continuing the above inequality we have for $x>r_d$ and $x\in[x_k]$
\begin{multline*}
\left(1+\frac{1}{2}\left(\frac{1}{Q(x_k)}-1\right)\right)\left|Q(x)\right|\leq\left(1+\frac{1}{2}\left(\frac{1}{Q(x_k)}-1\right)\right)Q(x_k)=\\
\frac{Q(x_k)+1}{2}<1
\end{multline*}
\item[Case III]If $r_d\in[x_k]$ set
\[\varepsilon_3=\frac{1}{2}\frac{1}{(r_d-r_1)\cdots(r_d-r_{d-1})}\]
and thus for $x=r_d$ and $x\in[x_k]$
\begin{align*}
|Q(x)+\varepsilon_3 (x-r_1)\cdots(x-r_{d-1})|=&\\
|Q(r_d)+\varepsilon_3 (r_d-r_1)\cdots(r_d-r_{d-1})|=&\left|Q(r_d)+\frac{1}{2}\right|=0+\frac{1}{2}<1
\end{align*}
Otherwise, if $r_d\notin[x_k]$, set
\[\varepsilon_3=1\]
\end{description}
Set $\varepsilon=\min\left\{\varepsilon_1,\varepsilon_2,\varepsilon_3\right\}$.  Using $\varepsilon$ for the inequalities in the above three cases gives us that
\[\left|\hat{Q}(x)\right|<1\]
for $x\in[x_k]$.  Clearly the lead coefficients of $\hat{Q}(x)$ and $Q(x)$ are equal since we defined $\hat{Q}(x)$ as $Q(x)$ plus a degree $d-1$ polynomial. But since $|\hat{Q}(x)|$ is strictly less than one we can have some $\lambda>1$ such that
\[\left|\lambda\hat{Q}(x)\right|\leq 1\]
for $x\in[x_k]$. But then $\lambda a_d$, the lead coefficient of $\lambda\hat{Q}$ is greater than $a_d$. This contradicts the maximality of the lead coefficient $a_d$ of $Q(x)$.  The same arguments work to show that $Q(x_1)=(-1)^d$.  
\end{proof}
Next we prove that $L_{d,[x_k]}$ passes through some point $(x_i,\pm1)$ between any two of its consecutive roots.

\begin{lemma}\label{Thm:TouchingBetweenRoots}
If for $k>d\geq2$
\[Q(x)=a_d\prod_{i\in [d]}(x-r_i)\]
is a degree $d$ polynomial with maximum lead coefficient $a_d$ such that
\[|Q(x)|\leq 1\]
when
\[x\in[x_k]=\{x_1< x_2<\cdots< x_k\}\]
then for any two of $Q(x)$'s consecutive roots $r_i$ and $r_{i+1}$ there exists an $x^\prime \in[x_k]$ such that 
\[r_i<x^\prime <r_{i+1}\]
and 
\[|Q(x^\prime)|=1.\] \end{lemma}

\begin{proof} Suppose that $|Q(x)|<1$ for any and all $x\in[x_k]$ where $r_i<x<r_{i+1}$.  We know that $Q(x)$ has $d$ distinct roots and that they are contained in the interval $(x_1,x_k)$, from Theorem \ref{Thm:DistinctRoots} and Lemma \ref{Thm:RootsInterval} respectively.  In other words, for some
\[x_1<r_1<r_2<\cdots<r_{d-1}<r_d<x_k,\]
we have
\[Q(x)=a_d(x-r_1)\cdots(x-r_d).\]
Now with some soon to be determined $\varepsilon>0$ we define 
\begin{align*}
\check{Q}(x)&=Q(x)-\varepsilon (x-r_1)\cdots (x-r_{i-1})(x-r_{i+2})\cdots(x-r_{d}).
\end{align*}
So for $x\notin\{r_i,r_{i+1}\}$ and $x\in[x_k]$ we can write
\[\check{Q}(x)=\left(1-\frac{\varepsilon}{a_d(x-r_i)(x-r_{i+1})}\right)Q(x).\]
We will use this $\check{Q}(x)$ to contradict the maximality of $Q(x)$'s lead coefficient.  We will choose $\varepsilon$ to be the minimum  of $\varepsilon_1$, $\varepsilon_2$, $\varepsilon_3$ and $\varepsilon_4$, which are now defined for four cases of $x$:

\begin{description}
\item[Case I] Set
\[\varepsilon_1=\frac{a_d}{2}(\min\left\{(x-r_i)(x-r_{i+1}):x\in[x_k]\text{ and }(x<r_i\text{ or }r_{i+1}<x)\right\})\]
and thus for ($x<r_1$ or $r_{i+1}<x$) and $x\in[x_k]$
\begin{multline*}
\left|\left(1-\frac{\varepsilon_1}{a_d(x-r_i)(x-r_{i+1})}\right)Q(x)\right|=\\
\left|\left(1-\frac{1}{2}\frac{(\min\left\{(x-r_i)(x-r_{i+1}):x\in[x_k]\text{ and }(x<r_i\text{ or }r_{i+1}<x)\right\})}{(x-r_i)(x-r_{i+1})}\right)Q(x)\right|<\left|Q(x)\right|\leq 1
\end{multline*}

\item[Case II] If there exists an $x\in[x_k]$ such that $r_i<x<r_{i+1}$ then first set \[Q_{\max}=\max\left\{|Q(x)|:x\in[x_k]\text{ and }r_i<x<r_{i+1}\right\}.\]
Recall that we assumed that $|Q(x)|<1$ for any and all $x\in[x_k]$ where $r_i<x<r_{i+1}$.  This means $0<Q_{\max}<1$.
Next set 
\[\varepsilon_2=\frac{a_d}{2}\left(\frac{1}{Q_{max}}-1\right)(\min\left\{-(x-r_i)(x-r_{i+1}):x\in[x_k]\text{ and }r_i<x<r_{i+1})\right\}).\]
\begin{multline*}
\left|\left(1-\frac{\varepsilon_2}{a_d(x-r_i)(x-r_{i+1})}\right)Q(x)\right|=\\
\left(1-\frac{1}{2}\frac{(\min\left\{(x-r_i)(x-r_{i+1}):x\in[x_k]\text{ and }r_i<x<r_{i+1})\right\})}{(x-r_i)(x-r_{i+1})}\left(\frac{1}{Q_{max}}-1\right)\right)\left|Q(x)\right|\leq\\
\left(1+\frac{1}{2}\left(\frac{1}{Q_{max}}-1\right)\right)\left|Q(x)\right|\leq\left(1+\frac{1}{2}\left(\frac{1}{Q_{max}}-1\right)\right)Q_{max}=\frac{Q_{max}+1}{2}<1
\end{multline*}
Otherwise, if there does not exist an $x\in[x_k]$ such that $r_i<x<r_{i+1}$, set $\varepsilon_2=1$

\item[Case III]If $r_i\in[x_k]$ set
\[\varepsilon_3=\frac{1}{2}\frac{1}{(r_i-r_1)\cdots (r_i-r_{i-1})(r_i-r_{i+2})\cdots(r_i-r_{d})}\]
and thus for $x=r_i$ and $x\in[x_k]$
\begin{align*}
|Q(x)+\varepsilon_3 (x-r_1)\cdots (x-r_{i-1})(x-r_{i+2})\cdots(x-r_{d})|=&\\
|Q(r_i)+\varepsilon_3 (r_i-r_1)\cdots (r_i-r_{i-1})(r_i-r_{i+2})\cdots(r_i-r_{d})|=&\left|Q(r_i)+\frac{1}{2}\right|=0+\frac{1}{2}<1.
\end{align*}
Otherwise, if $r_i\notin[x_k]$, set
\[\varepsilon_3=1\]

\item[Case IV]If $r_{i+1}\in[x_k]$ set
\[\varepsilon_4=\frac{1}{2}\frac{1}{(r_{i+1}-r_1)\cdots (r_{i+1}-r_{i-1})(r_{i+1}-r_{i+2})\cdots(r_{i+1}-r_{d})}\]
and thus for $x=r_{i+1}$ and $x\in[x_k]$
\begin{align*}
|Q(x)+\varepsilon_4 (x-r_1)\cdots (x-r_{i-1})(x-r_{i+2})\cdots(x-r_{d})|=&\\
|Q(r_{i+1})+\varepsilon_4 (r_{i+1}-r_1)\cdots (r_{i+1}-r_{i-1})(r_{i+1}-r_{i+2})\cdots(r_{i+1}-r_{d})|=&\left|Q(r_{i+1})+\frac{1}{2}\right|=0+\frac{1}{2}<1.
\end{align*}
Otherwise, if $r_{i+1}\notin[x_k]$, set
\[\varepsilon_4=1\]
\end{description}
Now set $\varepsilon=\min\{\varepsilon_1,\varepsilon_2,\varepsilon_3,\varepsilon_4\}$.  Using $\varepsilon$ in each of the above four cases yields
\[\left|\check{Q}(x)\right|<1\]
for $x\in[x_k]$.  Clearly the lead coefficients of $\check{Q}(x)$ and $Q(x)$ are equal since we have defined $\check{Q}(x)$ as $Q(x)$ plus a degree $d-2$ polynomial. But since $|\check{Q}(x)|$ is strictly less than one we can have some $\lambda>1$ such that
\[\left|\lambda\check{Q}(x)\right|\leq 1\]
for $x\in[x_k]$. But then then $\lambda a_d$, the lead coefficient of $\lambda\check{Q}$ is greater than $a_d$. This contradicts the maximality of the lead coefficient $a_d$ of $Q(x)$.
\end{proof}

Finally we prove that there is a unique $L_{d,[x_k]}(x)$ for every $[x_k]$ and $d$ whenever $k>d\geq1$.

\begin{theorem}\label{Thm:Unique} If for $k>d\geq2$
\[L_{d,[x_k]}(x)=a_d\prod_{i\in [d]}(x-r_i)\]
is a degree $d$ polynomial with maximum lead coefficient $a_d$ such that
\[|L_{d,[x_k]}(x)|\leq 1\]
when
\[x\in[x_k]=\{x_1< x_2<\cdots< x_k\}\]
then $L_{d,[x_k]}(x)$ is the \emph{UNIQUE} degree $d$ polynomial with maximum lead coefficient $a_d$ such that for a given
\[[x_k]=\{x_1< x_2<\cdots< x_k\}\]
we have that for $x\in[x_k]$
\[|L_{d,[x_k]}|\leq 1.\]\end{theorem}
\begin{proof} Assume that $\grave{Q}(x)$ and $\acute{Q}(x)$ are two degree $d$ polynomials with the same lead coefficient $a_d>0$ and satisfying our condition that $a_d$ is the maximum lead coefficient such that both
\[|\grave{Q}(x)|\leq 1\]
and
\[|\acute{Q}(x)|\leq 1\]
when $x\in[x_k]$.
Now form the average  
\[\bar{Q}(x)=\frac{\grave{Q}(x)+\acute{Q}(x)}{2}.\]
This average also clearly has lead coefficient $a_d$ and satisfies our condition since
\[|\bar{Q}(x)|=\left|\frac{\grave{Q}(x)+\acute{Q}(x)}{2}\right|\leq \frac{|\grave{Q}(x)|+|\acute{Q}(x)|}{2}\leq 1\]
when $x\in[x_k]$.
By Theorem \ref{Thm:DistinctRoots}, Lemma \ref{Thm:TerminalPoints} and Lemma \ref{Thm:TouchingBetweenRoots}, for some 
\[a_1<a_2<\cdots<a_{d-1}\] 
such that 
\[\{a_1,a_2,\cdots,a_{d-1}\}\subset\{x_2,x_3,\cdots,x_{k-1}\}\] 
and with
\[\{b_1,b_2,\cdots,b_{d+1}\}=\{x_1,a_1,a_2,\cdots a_{d-1},x_k\}\] 
we have that the polynomial $\bar{Q}(x)$ passes through the $d+1$ points
\[(b_i,(-1)^{(d+1)-i})\text{ for }i\in[d+1]\]
or, stated differently, that
\[\bar{Q}(b_i)=(-1)^{(d+1)-i}\]
for $i\in[d+1]$.
But then because of our conditions on $\grave{Q}(x)$ and $\acute{Q}(x)$ bounding them between $-1$ and $1$ for $x\in[x_k]$ and since $\bar{Q}(x)$ is an average of these two polynomials, we must also have that
\[\grave{Q}(b_i)=\acute{Q}(b_i)=(-1)^{(d+1)-i}\]
for $i\in[d+1]$.  But if the two degree $d$ polynomials $\grave{Q}(x)$ and $\acute{Q}(x)$ are equal at $d+1$ points then they are equal everywhere.  Thus 
\[L_{d,[x_k]}(x)=\grave{Q}(x)=\acute{Q}(x)=\bar{Q}(x).\]
\end{proof}
\end{section}

\begin{section}{Calculating the Lead Coefficients of $L_{d,[k]}$ for $d\leq 4$ and an Algorithm for Generating All $L_{d,[x_k]}$}\label{Sec:LeadCoefficients}

This section begins with the algebraic computation of the lead coefficients for some specific examples of $L_{d,[x_k]}$.  Namely, when $[x_k]$ is the arithmetic progression $[k]=\{1,\cdots,k\}$.  Recall from Definition \ref{Def:LPolynomials} that $L_{d,[x_k]}$ is the unique degree $d$ polynomial with maximum lead coefficient when bounded between $-1$ and $1$ on the set $[k]$.  
This section ends with the description of an algorithm for generating all $L_{d,[x_k]}$ (and thus their lead coefficients).

\begin{subsection}{Lead coefficients of $L_{1,[k]}$}
This case is simple.  Lemma \ref{Thm:TerminalPoints} forces $L_{1,[k]}(1)=-1$ and $L_{1,[k]}(k)=1$.  Since $L_{1,[k]}(x)$ is a line and we have two points lying on it we find
\[L_{1,[k]}(x)=a_1x+a_0=\frac{2}{k-1}x-\frac{k+1}{k-1}.\] 
Thus the lead coefficients of $L_{1,[k]}(x)$
\[a_1=\frac{2}{k-1}.\]
\end{subsection}

\begin{subsection}{Lead coefficients of $L_{2,[k]}$}
Recall that $L_{2,[k]}(x)$ is the unique (by Theorem \ref{Thm:Unique}) degree $2$ polynomial bounded between $-1$ and $1$ for $x\in[k]=\{1,\cdots,k\}$ with maximum lead coefficient.  We shift $L_{2,[k]}(x)$ so that the $k$ consecutive $x$ values it is bounded on are centered at zero (that is, instead of being bounded for $x\in\{1,\cdots,k\}=[k]$ it is bounded for $x\in\{-\frac{k-1}{2},\cdots,\frac{k-1}{2}\}$).  This alters neither the shape of the polynomial in general nor its lead coefficient.  We set 
\[k^\prime=\left\{-\frac{k-1}{2},\cdots,\frac{k-1}{2}\right\}\]
and thus $L_{2,k^\prime}(x)$ refers to this shifted polynomial which, we will see, minimizes the number of unknown coefficients. Observe that since $\deg(L_{2,k^\prime}(x))=2$ is even, the difference 
\[\frac{L_{2,k^\prime}(x)+L_{2,k^\prime}(-x)}{2}=a_2 x^2+a_0\] 
describes an even-function that has the same lead coefficient as $L_{2,k^\prime}(x)$.
Since $|L_{2,k^\prime}(x)|\leq 1$ for $x\in\{-\frac{k-1}{2},\cdots,\frac{k-1}{2}\}$ then also $\left|\frac{L_{2,[k]}(x)+L_{2,k^\prime}(-x)}{2}\right|\leq 1$  for $x\in\{-\frac{k-1}{2},\cdots,\frac{k-1}{2}\}$.  But by the uniqueness of $L_{2,k^\prime}(x)$ (from Theorem \ref{Thm:Unique}) we must have
\[L_{2,k^\prime}(x)=\frac{L_{2,k^\prime}(x)+L_{2,k^\prime}(-x)}{2}\]
and thus
\[L_{2,k^\prime}(x)=a_2 x^2+a_0.\]
Two unknown coefficients are better than three.  
From Lemma \ref{Thm:TerminalPoints}, $L_{2,k^\prime}$ passes through $(\frac{k-1}{2},1)$.  Thus
\[L_{2,k^\prime}\left(\frac{k-1}{2}\right)=a_2\left(\frac{k-1}{2}\right)^2+a_0=1.\]  
Solving for $a_0$ gives
\[a_0=1-a_2\left(\frac{k-1}{2}\right)^2\]
Plugging this into 
\begin{align*}
a_2x^2+a_0 &\leq 1\\
a_2x^2+a_0 &\geq -1,
\end{align*}
the constraining inequalities on $L_{2,k^\prime}(x)$ for $x\in\{-(\frac{k-1}{2}-1),\cdots,\frac{k-1}{2}-1\}$,
we get
\begin{eqnarray}
a_2\left(x^2-\left(\frac{k-1}{2}\right)^2\right)&\leq&0\label{Ineq:VacuousBound1}\\
a_2&\leq&\frac{8}{(k-1)^2-4x^2}\label{Ineq:UpperBound1}
\end{eqnarray}
for $x\in\{-(\frac{k-1}{2}-1),\cdots,\frac{k-1}{2}-1\}$.

Inequality (\ref{Ineq:VacuousBound1}) is not useful because it is always true, since $a_2>0$ and $x^2<\left(\frac{k-1}{2}\right)^2$.  Inequality (\ref{Ineq:UpperBound1}) however gives an upper bound on $a_2$.  The right-side of inequality (\ref{Ineq:UpperBound1}) is minimum for $x\in\{-(\frac{k-1}{2}-1),\cdots,\frac{k-1}{2}-1\}$ when $x$ is closest or equal to zero.  For odd $k$ it is minimum when $x=0$.  For even $k$ it is minimum when $x=\pm\frac{1}{2}$.
Plugging the minimizing values of $x$ into inequality (\ref{Ineq:UpperBound1}) yields
\[
a_2=
\begin{cases}
\dfrac{8}{(k-1)^2}&\text{for }k\equiv 1\bmod 2\\
\dfrac{8}{k(k-2)}&\text{for }k\equiv 0\bmod 2
\end{cases}
\]
as the lead coefficient for $L_{2,k^{\prime}}(x)$ and thus also for $L_{2,[k]}(x)$.  
\end{subsection}

\begin{subsection}{Lead coefficients of $L_{3,[k]}$}
Recall that $L_{3,[k]}(x)$ is the unique (by Theorem \ref{Thm:Unique}) degree $3$ polynomial that while bounded between $-1$ and $1$ for $x\in[k]=\{1,\cdots,k\}$ has maximum possible lead coefficient.  We shift $L_{3,[k]}(x)$ so that, as with $L_{2,[k]}(x)$ above, the $k$ consecutive $x$ values it is bounded on are centered at zero (that is, instead of being bounded for $x\in\{1,\cdots,k\}$ it is bounded for $x\in\{-\frac{k-1}{2},\cdots,\frac{k-1}{2}\}$). This alters neither the shape of the polynomial in general nor its lead coefficient.  Again
\[k^\prime=\left\{-\frac{k-1}{2},\cdots,\frac{k-1}{2}\right\}.\]
So $L_{3,k^\prime}(x)$ refers specifically to this shifted polynomial which will minimize the number of unknown coefficients.  Observe that since $\deg(L_{3,k^\prime}(x))=3$ is odd, the difference
\[\frac{L_{3,k^\prime}(x)-L_{3,k^\prime}(-x)}{2}=a_3 x^3+a_1 x\] 
describes an odd-function that has the same lead coefficient as $L_{3,k^\prime}(x)$.
Since $|L_{3,k^\prime}(x)|\leq 1$ for $x\in\{-\frac{k-1}{2},\cdots,\frac{k-1}{2}\}$ then also $\left|\frac{L_{3,k^\prime}(x)-L_{3,k^\prime}(-x)}{2}\right|\leq 1$  for $x\in\{-\frac{k-1}{2},\cdots,\frac{k-1}{2}\}$.  But then by uniqueness (from Theorem \ref{Thm:Unique}) we must have
\[L_{3,k^\prime}(x)=\frac{L_{3,k^\prime}(x)-L_{3,k^\prime}(-x)}{2}\]
and thus
\[L_{3,k^\prime}(x)=a_3 x^3+a_1 x.\]
Two unknown coefficients are better than four.  

From Lemma \ref{Thm:TerminalPoints}, $L_{3,k^\prime}$ passes through $(\frac{k-1}{2},1)$.  Thus
\[L_{3,k^\prime}\left(\frac{k-1}{2}\right)=a_3\left(\frac{k-1}{2}\right)^3+a_1\left(\frac{k-1}{2}\right)=1.\]  
Solving for $a_1$ gives
\[a_1=\frac{8-a_3 (k-1)^3}{4(k-1)}\]
Plugging this into 
\begin{align*}
a_3x^3+a_1x^1&\leq 1\\
a_3x^3+a_1x^1&\geq -1,
\end{align*}
the constraining inequalities on $L_{3,k^\prime}(x)$ for $x\in\{-(\frac{k-1}{2}-1),\cdots,\frac{k-1}{2}-1\}$, we get
\begin{align}
a_3\geq\frac{-4}{2(k-1)x^2+(k-1)^2x}\label{Ineq:LowerBound}\\
a_3\leq\frac{-4}{2(k-1)x^2-(k-1)^2x}\label{Ineq:UpperBound}
\end{align}
for $x\in\{-(\frac{k-1}{2}-1),\cdots,\frac{k-1}{2}-1\}$.

Inequality (\ref{Ineq:LowerBound}) is not useful.  It gives lower bounds for our positive $a_3$ but we are looking for the maximum possible value of $a_3$.  Inequality (\ref{Ineq:UpperBound}) however gives us upper bounds for $a_3$.  The right-side of inequality (\ref{Ineq:UpperBound}) is minimum on $x\in\{-(\frac{k-1}{2}-1),\cdots,\frac{k-1}{2}-1\}$ for $x$ closest or equal to $\frac{k-1}{4}$.
Plugging the minimizing values of $x$ into the inequality (\ref{Ineq:UpperBound}) yields
\[
a_3=
\begin{cases}
\dfrac{32}{(k-1)^3}&\text{when } k\equiv 1\bmod 4,\text{ set } x=\frac{k-1}{4}\\
\dfrac{32}{k(k-1)(k-2)}&\text{when } k\equiv 2\bmod 4,\text{ set } x=\frac{k}{4}\\
\dfrac{32}{(k+1)(k-1)(k-3)}&\text{when } k\equiv 3\bmod 4,\text{ set } x=\frac{k-3}{4}\text{ or } x=\frac{k+1}{4}\\
\dfrac{32}{k(k-1)(k-2)}&\text{when } k\equiv 0\bmod 4,\text{ set } x=\frac{k-2}{4}
\end{cases}
\]
for the lead coefficient of $L_{3,k^{\prime}}(x)$ and thus also for $L_{3,[k]}(x)$.  
\end{subsection}

\begin{subsection}{Lead coefficients of $L_{4,[k]}$}
Using the same arguments as from $L_{2,[k]}$ above we know that the maximum possible lead coefficient $a_4 > 0$ of the polynomial and even-function
	\[L_{4,k^\prime}(x)=a_4 x^4+a_2 x^2 + a_0,\]
where $k^\prime= \{-\frac{k-1}{2},\cdots,\frac{k-1}{2}\}$ and $k>4$, is the same as the maximum possible lead coefficient of $L_{4,[k]}$.

This $d=4$ case is itself split into two cases, one for odd $k$ and one for even $k$:

\begin{description}
\item[Case I] For odd $k$. By Theorem \ref{Thm:DistinctRoots} and \ref{Thm:TouchingBetweenRoots} we have the equation
\[L_{4,k^\prime}(0)=1\]
and so
\[a_0=1.\]

By Lemma \ref{Thm:TerminalPoints} we have the equation
\[L_{4,k^\prime}\left(\frac{k-1}{2}\right)=1\]
and thus
\[a_4\left(\frac{k-1}{2}\right)^4+a_2 \left(\frac{k-1}{2}\right)^2 + a_0=1\]
Rewriting this in terms of $a_2$ after substituting $1$ for $a_0$ gives us
\[a_2=-a_4\left(\frac{k-1}{2}\right)^2\]
Recall that the lower bound is
\[L_{4,k^\prime}(x)\geq -1\]
for $x\in \{-(\frac{k-1}{2}),\cdots,\frac{k-1}{2}\}$.  Plugging the results from above into this inequality gives
\[a_4 x^4 - a_4\left(\frac{k-1}{2}\right)^2x^2 +1 \geq -1\]
and thus
\[a_4\leq \frac{8}{x^2(k-1)^2-4x^4}\]
for $x\in \{-(\frac{k-1}{2}),\cdots,\frac{k-1}{2}\}$.

Optimization of the right-side of this last inequality on the interval $(0,\frac{k-1}{2})$ gives a minimum at $x=\frac{k-1}{2\sqrt{2}}$. We pick the neighboring integer point that gives the smallest upper bound: so for odd $k$ the maximum lead coefficient of $L_{4,k^{\prime}}(x)$, and thus also of $L_{4,[k]}(x)$, is  
\[a_4=\min\limits_{ x\in I}\left\{\frac{8}{x^2(k-1)^2-4x^4}\right\}\]
where $I=\{\lceil {\frac{k-1}{2\sqrt{2}}}\rceil,\lfloor {\frac{k-1}{2\sqrt{2}}}\rfloor\}$.\\

\item[Case II]
For even $k$. Again by Theorem \ref{Thm:DistinctRoots} and Lemma \ref{Thm:TouchingBetweenRoots} we have
\[L_{4,k^\prime}\left(\frac{1}{2}\right)=1\]
which yields
\[L_{4,k^\prime}\left(\frac{1}{2}\right)=a_4 \left(\frac{1}{2}\right)^4+a_2 \left(\frac{1}{2}\right)^2 + a_0=1\]
and thus 
\begin{equation}\label{d4keven1}
a_0=1-\frac{a_2}{4}-\frac{a_4}{16}
\end{equation}
And again by Lemma \ref{Thm:TerminalPoints} we have the second equation
\[L_{4,k^\prime}\left(\frac{k-1}{2}\right)=1\]
and thus
\begin{equation}\label{d4keven2}
a_4\left(\frac{k-1}{2}\right)^4+a_2\left(\frac{k-1}{2}\right)^2 + a_0=1
\end{equation}
and so by combining the two equations (\ref{d4keven1}) and (\ref{d4keven2}) we have
\[a_0=a_4\frac{(k-1)^2}{16}+1\]
\[a_2=-a_4\frac{(k-1)^2+1}{4}.\]
Recall that the lower bound on our polynomial is
	\[L_{4,[k]}(x)\geq -1\]
for $x\in \{-(\frac{k-1}{2}),\cdots,\frac{k-1}{2}\}$.  Plugging the results from above section into this inequality yields
\[a_4\leq \frac{-32}{16x^4-4((k-1)^2+1)x^2+(k-1)^2}\]
for $x\in \{-(\frac{k-1}{2}),\cdots,\frac{k-1}{2}\}$

Optimization of the right-side of this last inequality on $(\frac{1}{2},\frac{k-1}{2})$ gives a minimum at $x=\frac{k-1}{2\sqrt{2}}$.  We pick the neighboring HALF-integer point that gives the smallest upper bound.  So for even $k$ the lead coefficients of $L_{4,k^{\prime}}(x)$, and thus also of $L_{4,[k]}(x)$, is
\[a_4=\min\limits_{ x\in H}\left\{\frac{-32}{16x^4-4((k-1)^2+1)x^2+(k-1)^2}\right\}\]
where
$H=\{\frac{1}{2}\lceil {\frac{k-1}{\sqrt{2}}}\rceil,\frac{1}{2}\lfloor {\frac{k-1}{\sqrt{2}}}\rfloor\}$.
\end{description}

\end{subsection} 

\begin{subsection}{An algorithm generating all $L_{d,[x_k]}$}\label{algorithmSection}
For $d=5$ only one of the three coefficients, $\{a_5,a_3,a_1\}$, can be eliminated when using the algebraic techniques from the sections on $d\leq 4$.  For $d=6$ only two of the four coefficients, $\{a_6,a_4,a_2,a_0\}$, can be eliminated. Fortunately, putting together some of our theorems and lemmas yields an algorithm that generates $L_{d,[x_k]}(x)$ for all $d\geq1$ and $k>d$. 

We begin similarly to Theorem \ref{Thm:Unique}.  For a given
\[[x_k]=\{x_1< x_2<\cdots< x_k\}\]
if $L_{d,[x_k]}(x)$ is the degree $d$ polynomial with maximum lead coefficient $a_d$ such that when $x\in[x_k]$
\[|L_{d,[x_k]}(x)|\leq 1\] 
then by Theorem \ref{Thm:DistinctRoots}, Lemma \ref{Thm:TerminalPoints} and Lemma \ref{Thm:TouchingBetweenRoots} there is some 
\[A=\{a_1<a_2<\cdots<a_{d-1}\}\subset\{x_2,x_3,\cdots,x_{k-1}\},\]  
and thus
\[\{b_1,b_2,\cdots,b_{d+1}\}=\{x_1,a_1,a_2,\cdots,a_{d-1},x_k\},\] 
such that the polynomial $L_{d,[x_k]}(x)$ passes through the $d+1$ points
\[(b_i,(-1)^{(d+1)-i})\text{ for }i\in[d+1].\]  
But now since $L_{d,[x_k]}(x)$ is a degree $d$ polynomial we could solve for it explicitly if we knew the set $A$.  With a computer this is easily done. For each of the $\binom{k-2}{d-1}$ possible $d-1$ element sets $A\subset\{x_2,x_3,\cdots,x_{k-1}\}$, solve for the polynomial passing through the corresponding points
\[(b_i,(-1)^{(d+1)-i})\text{ for }i\in[d+1].\]  
From the multiset of all $\binom{k-2}{d-1}$ such polynomials choose the subset of polynomials bounded between $-1$ and $1$ for $x\in[x_k]$.  From this subset the polynomial with maximum lead coefficient is our $L_{d,[x_k]}(x)$ and its maximum lead coefficient is our $a_d$.  In particular, setting 
\[\{x_1,x_2,\cdots,x_{k}\} = \{1,2,\cdots,k\}\] 
yields $L_{d,[k]}(x)$.
\end{subsection}
\end{section}

\begin{section}{Some $L_{d,[k]}$ in Terms of Chebyshev Polynomials}\label{Sec:TermsCheb}
Now we write some of the $L_{d,[k]}(x)$ in terms of corresponding Chebyshev $T_d(x)$ polynomials.  We compose our $L_{d,[k]}(x)$ with 
\[t(x)=\frac{k-1}{2}x+\frac{k+1}{2}\] 
in order to fit the points $[k]=\{1,...,k\}$ (on which our $L_{d,[k]}(x)$ are bounded between $-1$ and $1$) into the interval $(-1,1)$ (on which Chebyshev's $T_d(x)$ are bounded between $-1$ and $1$).
\subsection{$L_{1,[k]}$ in terms of $T_1$}
\[L_{1,[k]}\left({ \frac{k-1}{2}x+\frac{k+1}{2}}\right)=x=T_1(x)\]

\subsection{$L_{2,[k]}$ in terms of $T_2$}
\begin{align*}
L_{2,[k]}\left({ \frac{k-1}{2}x+\frac{k+1}{2}}\right)&= T_2(x) &\scriptstyle\text{for k }\equiv 1\bmod 2\\
L_{2,[k]}\left({ \frac{k-1}{2}x+\frac{k+1}{2}}\right)&= T_2(x)+\frac{2}{k(k-2)}(x^2-1) &\scriptstyle\text{for k }\equiv 0\bmod 2
\end{align*}

\subsection{$L_{3,[k]}$ in terms of $T_3$}
\begin{align*}
L_{3,[k]}\left({\frac{k-1}{2}x+\frac{k+1}{2}}\right)&= T_3(x) &\scriptstyle\text{for k }\equiv 1\bmod 4\\
L_{3,[k]}\left({ \frac{k-1}{2}x+\frac{k+1}{2}}\right)&= T_3(x) + \frac{4}{k(k-2)}(x^3-x) &\scriptstyle\text{for k }\equiv 2\bmod 4\\
L_{3,[k]}\left({ \frac{k-1}{2}x+\frac{k+1}{2}}\right)&= T_3(x) + \frac{16}{(k+1)(k-3)}(x^3-x) &\scriptstyle\text{for k }\equiv 3\bmod 4\\
L_{3,[k]}\left({ \frac{k-1}{2}x+\frac{k+1}{2}}\right)&= T_3(x) + \frac{4}{k(k-2)}(x^3-x) &\scriptstyle\text{for k }\equiv 0\bmod 4
\end{align*}

\subsection{$L_{4,[k]}$ in terms of $T_4$}
For the $d=4$ case the method of solving for the lead coefficient $a_4$ does not directly yield a modular pattern like it did for the $d\in\{1,2,3\}$ cases.  For $21\geq k\geq 5$ we have 
\begin{align*}
L_{4,[5]}\left({ \frac{4}{2}x+\frac{6}{2}}\right)&= T_4(x) + \frac{8}{3}(x^4-x^2) \\
%L_{4,[6]}\left({ \frac{(k-1)}{2}x+\frac{(k+1)}{2}}\right)=L_{4,6^{\prime\prime}}(x)&= T_4(x) + \frac{113}{64}x^4-\frac{69}{32}x^2+\frac{25}{64}\\
L_{4,[6]}\left({ \frac{5}{2}x+\frac{7}{2}}\right)&= T_4(x) + \frac{1}{64}(x^2-1)(113x^2-25)\\
L_{4,[7]}\left({ \frac{6}{2}x+\frac{8}{2}}\right)&= T_4(x) + \frac{1}{10}(x^4-x^2)\\
%L_{4,[8]}\left({ \frac{(k-1)}{2}x+\frac{(k+1)}{2}}\right)=L_{4,8^{\prime\prime}}(x)&= T_4(x) + \frac{97}{288}x^4-\frac{73}{144}x^2+\frac{49}{288}\\
L_{4,[8]}\left({ \frac{7}{2}x+\frac{9}{2}}\right)&= T_4(x) + \frac{1}{288}(x^2-1)(97x^2-49)\\
L_{4,[9]}\left({ \frac{8}{2}x+\frac{10}{2}}\right)&= T_4(x) + \frac{8}{63}(x^4-x^2)\\
%L_{4,[10]}\left({ \frac{(k-1)}{2}x+\frac{(k+1)}{2}}\right)=L_{4,10^{\prime\prime}}(x)&= T_4(x) + \frac{139}{256}x^4-\frac{83}{128}x^2+\frac{27}{256}\\
L_{4,[10]}\left({ \frac{9}{2}x+\frac{11}{2}}\right)&= T_4(x) + \frac{1}{256}(x^2-1)(139x^2-27)\\
L_{4,[11]}\left({ \frac{10}{2}x+\frac{12}{2}}\right)&= T_4(x) + \frac{49}{72}(x^4-x^2)\\
%L_{4,[12]}\left({ \frac{(k-1)}{2}x+\frac{(k+1)}{2}}\right)=L_{4,12^{\prime\prime}}(x)&= T_4(x) + \frac{817}{1728}x^4-\frac{469}{864}x^2+\frac{121}{1728}\\
L_{4,[12]}\left({ \frac{11}{2}x+\frac{13}{2}}\right)&= T_4(x) + \frac{1}{1728}(x^2-1)(817x^2-121)\\
L_{4,[13]}\left({ \frac{12}{2}x+\frac{14}{2}}\right)&= T_4(x) + \frac{1}{10}(x^4-x^2)\\
%L_{4,[14]}\left({ \frac{(k-1)}{2}x+\frac{(k+1)}{2}}\right)=L_{4,14^{\prime\prime}}(x)&= T_4(x) + \frac{401}{3520}x^4-\frac{57}{352}x^2+\frac{169}{3520}\\
L_{4,[14]}\left({ \frac{13}{2}x+\frac{15}{2}}\right)&= T_4(x) + \frac{1}{3520}(x^2-1)(401x^2-169)\\
L_{4,[15]}\left({ \frac{14}{2}x+\frac{16}{2}}\right)&= T_4(x) + \frac{1}{300}(x^4-x^2)\\
%L_{4,[16]}\left({ \frac{(k-1)}{2}x+\frac{(k+1)}{2}}\right)=L_{4,16^{\prime\prime}}(x)&= T_4(x) + \frac{47}{416}x^4-\frac{31}{208}x^2+\frac{15}{416}\\
\end{align*}
\begin{align*}
L_{4,[16]}\left({ \frac{15}{2}x+\frac{17}{2}}\right)&= T_4(x) + \frac{1}{416}(x^2-1)(47x^2-15)\\
L_{4,[17]}\left({ \frac{16}{2}x+\frac{18}{2}}\right)&= T_4(x) + \frac{8}{63}(x^4-x^2)\\
%L_{4,[18]}\left({ \frac{(k-1)}{2}x+\frac{(k+1)}{2}}\right)=L_{4,18^{\prime\prime}}(x)&= T_4(x) + \frac{2881}{10080}x^4-\frac{317}{1008}x^2+\frac{289}{10080}\\
L_{4,[18]}\left({ \frac{17}{2}x+\frac{19}{2}}\right)&= T_4(x) + \frac{1}{10080}(x^2-1)(2881x^2-289)\\
L_{4,[19]}\left({ \frac{18}{2}x+\frac{20}{2}}\right)&= T_4(x) + \frac{1}{10}(x^4-x^2)\\
%L_{4,[20]}\left({ \frac{(k-1)}{2}x+\frac{(k+1)}{2}}\right)=L_{4,20^{\prime\prime}}(x)&= T_4(x) + \frac{1297}{16128}x^4-\frac{829}{8064}x^2+\frac{361}{16128}\\
L_{4,[20]}\left({ \frac{19}{2}x+\frac{21}{2}}\right)&= T_4(x) + \frac{1}{16128}(x^2-1)(1297x^2-361)\\
L_{4,[21]}\left({ \frac{20}{2}x+\frac{22}{2}}\right)&= T_4(x) + \frac{8}{2499}(x^4-x^2).\\
\end{align*}

\end{section}


\begin{thebibliography}{99}

\bibitem[OB10]{O'Bryant} K. O'Bryant, Sets of integers that do not contain long arithmetic progressions (June 21,2010), available at http://arxiv.org/abs/0811.3057v3
\bibitem[Ri74]{Rivlin} T.J. Rivlin, The Chebyshev polynomials. Pure and Applied Mathematics. Wiley-Interscience [John Wiley \& Sons], New York-London-Sydney,1974. Chapter 2, "Extremal Properties", pp. 56--123.
\bibitem[Ro53]{Roth} K.F. Roth, On certain sets of integers, {\em J. London Math. Soc.} \textbf{28} (1953), 104-109.
\bibitem[Sz90]{Szemeredi} E. Szemer\'{e}di, Integer sets containing no arithmetic progressions, {\em Acta Math. Hung.} \textbf{56} (1-2) (1990), 155-158.
%\bibitem[HB87]{Szemeredi} D.R. Heath-Brown, Integer sets containing no arithmetic progressions, {\em J. London Math. Soc.} \textbf{2} 35 (1987), 385-394.

\end{thebibliography}
\end{document}